\UseRawInputEncoding
\documentclass[11pt]{article}
\usepackage{bbm}
\usepackage[top=1in, bottom=1.5in]{geometry}
\usepackage{color}
\usepackage{comment}
\usepackage{amsfonts}
\usepackage{latexsym, amssymb, amsmath, amscd, amsthm, amsxtra}
\usepackage{mathtools}
\usepackage{enumerate}
\usepackage[all]{xy}
\usepackage{mathrsfs}
\usepackage{fancyhdr}
\usepackage{listings}
\usepackage{hyperref}
\usepackage{float}
\usepackage{algorithm}
\usepackage{algpseudocode}
\usepackage{framed}
\usepackage{lipsum}
\usepackage{stmaryrd}
\usepackage{dsfont}

\hypersetup{
	colorlinks   = true, %Colours links instead of ugly boxes
	urlcolor     = blue, %Colour for external hyperlinks
	linkcolor    = blue, %Colour of internal links
	citecolor   = red %Colour of citations
}

\newtheorem{theorem}{Theorem}[section]
\newtheorem{lemma}[theorem]{Lemma}

\newtheorem{proposition}[theorem]{Proposition}

\theoremstyle{definition}
\newtheorem{definition}[theorem]{Definition}

\usepackage{tocloft}
\usepackage{times}

\newcommand{\ER}{Erd\H{o}s-R\'enyi\ }

\newcommand{\HD}[1]{\textcolor{cyan}{[$\mathsf{HD}$:\ #1]}}

\newcommand{\llb}{\llbracket}
\newcommand{\rrb}{\rrbracket}
\newcommand{\pa}{\operatorname{P\!A}_{n}^{(m,\delta)}}

\theoremstyle{definition}
\numberwithin{equation}{section}

\title{Asymptotic diameter of preferential attachment model}

\author{Hang Du\thanks{Department of Mathematics, Massachusetts Institute of Technology} \and Shuyang Gong\thanks{School of Mathematical Sciences, Peking University} \and Zhangsong Li\footnotemark[2] \and Haodong Zhu\thanks{Department of Mathematics and Computer Science, Eindhoven University of Technology}}

\newcommand{\PA}{\operatorname{P\!A}}

\date{}

\begin{document}
	
\maketitle

\begin{abstract}
We study the asymptotic diameter of the preferential attachment model $\PA_n^{(m,\delta)}$ with parameters $m \ge 2$ and $\delta > 0$. Building on the recent work \cite{VZ25}, we prove that the diameter of $G_n \sim \PA_n^{(m,\delta)}$ is $(1+o(1))\log_\nu n$ with high probability, where $\nu$ is the exponential growth rate of the local weak limit of $G_n$. Our result confirms the conjecture in \cite{VZ25} and closes the remaining gap in understanding the asymptotic diameter of preferential attachment graphs with general parameters $m \ge 1$ and $\delta >-m$. Our proof follows a general recipe that relates the diameter of a random graph to its typical distance, which we expect to have applicability in a broader range of models.
\end{abstract}

\section{Introduction}
The \emph{preferential attachment model} is one of the mostly studied randomly growing network models. Given a parameter $m \in \mathbb{N}$, a preferential attachment graph on the vertex set $\{v_1, \dots, v_n\}$ with $m$ attachments is generated via the following iterative process: for each $t \ge 2$, the new vertex $v_t$ connects $m$ (not necessarily distinct) edges to vertices in the existing graph on $\{v_1, \dots, v_{t-1}\}$. Each endpoint of these edges is chosen independently, according to a probability distribution that favors vertices of higher degree in the current graph. In the classical setting, which is also the focus of this paper, the attachment probability is taken to be proportional to an \emph{affine function} of the degree, parameterized by $\delta > -m$ (see Definition~\ref{def-preferential-attachment-model} for the precise formulation). We denote by $\PA_n^{(m, \delta)}$ the distribution of the resulting random graph on $n$ vertices under this model. 

Our main result provides an asymptotic characterization of the diameter (i.e., the maximal distance between vertex pairs) of $G_n \sim \PA_n^{(m,\delta)}$, where the parameters satisfy $m \ge 2$ and $\delta > 0$.

\begin{theorem}\label{thm-main}
    Fix any $m \ge 2$ and $\delta > 0$. Then for $G_n \sim \PA_{n}^{(m,\delta)}$, it holds that
    \[
    \frac{\operatorname{diam}(G_n)}{\log_\nu n} \stackrel{\text{in probability}}{\longrightarrow} 1 \,, \quad \text{as } n \to \infty \,,
    \]
    where $\nu$ is the exponential growth rate of the local weak limit of $\PA_n^{(m,\delta)}$, as defined in~\eqref{eq-nu} below.
\end{theorem}

Theorem~\ref{thm-main} confirms a recent conjecture posed in~\cite{VZ25}, and closes the remaining gap in the asymptotic understanding of the diameter of $G_n \sim \PA_n^{(m,\delta)}$ for general parameters $m \ge 1$ and $\delta > -m$ (see Section~\ref{subsec-background} for further backgrounds).

\subsection{Backgrounds and Related work}\label{subsec-background}

The \emph{preferential attachment model}, originally introduced in~\cite{BA99}, is designed to capture the structural properties of many real-world networks. Since its inception, it has found wide-ranging applications in modeling and analyzing diverse types of networked systems. Examples include the World Wide Web~\cite{AH00, KBM13}, scientific collaboration and citation networks~\cite{New01, PGFSM08, Csa06, WYY08}, as well as many other social networks~\cite{capocci2006preferential,de2007preferential}. We refer to \cite[Section~1]{VZ25} for a more comprehensive overview on preferential attachment model and its relevance to various intriguing aspects of real-world networks.

A striking feature commonly observed in the study of large-scale networks is the \emph{small-world phenomenon}, which refers to the empirical observation that the diameter (or typical distance between vertex pairs) remains surprisingly small, often growing only logarithmically with the network size. A well-known example is the ``six degrees of separation'' principle, which posits that any two individuals on Earth are connected by a chain of at most six acquaintances. Motivated by the desire to understand this phenomenon from a theoretical standpoint, considerable attention has been devoted to studying the diameter of random graphs that model real-world structures. For classical random graph models such as the \ER graphs and the random $d$-regular graphs, the asymptotic behavior of the diameter is well understood; see, e.g., \cite{Bol82, CL01, RW10, Shi18} and the references therein. In the case of the preferential attachment model, given its simplicity and broad applicability, understanding the asymptotic behavior of the diameter emerges as a natural and important problem that has attracted significant interest.

The asymptotic diameter of the preferential attachment model has been well understood in certain parameter regimes. When $m = 1$ and $\delta > -1$, it is known from~\cite[Theorem~1]{Pittle94} that the diameter of $G_n \sim \operatorname{PA}^{(1,\delta)}_n$ is typically given by\footnote{In this paper, we adopt the Bachmann-Landau family of notations to characterize the order of approximation.}
\[
(1 + o(1)) \, \tfrac{2(1 + \delta)\log n}{(2 + \delta)\theta} \,,
\]
where $\theta \in (0,1)$ is the solution to $\theta + (1 + \delta)(1 + \log \theta) = 0$. Additionally, when $m \geq 2$ and $-m<\delta<0$, it is shown in~\cite[Theorem~1.3]{CGvdH19} that the diameter of $G_n$ is typically
\[
(1 + o(1)) \left( \tfrac{4}{|\log(1 + \delta/m)|} + \tfrac{2}{\log m} \right) \log\log n \,.
\]
Finally, when $m \geq 2$ and $\delta = 0$, it is shown in~\cite[Theorem~1]{BR09} that a variant (which allows self-loops) of
%\HD{How much does it differ to the one we are considering?}
the preferential attachment model has typical diameter
\[
(1 + o(1)) \, \frac{\log n}{\log\log n} \,.
\]
%\HD{We should make sure we find the original works of the above theorems and cite them.}

Despite the aforementioned advancements in understanding the asymptotic diameter of preferential attachment models, the case of $G_n \sim \PA_n^{(m,\delta)}$ with $m \ge 2$ and $\delta > 0$ has remained open. Prior to our work, the best known result in this regime asserted only that the diameter of $G_n$ is typically $O(\log n)$~\cite[Theorem~8.33]{vdHof24}. On the other hand, a very recent work~\cite{VZ25} by van der Hofstad and the last author establishes that the typical distance (i.e., the graph distance between two uniformly chosen vertices) in $G_n \sim \PA_n^{(m,\delta)}$ is approximately $\log_\nu n$, where
\begin{equation}\label{eq-nu}
\nu = \frac{2m(m + \delta) + 2\sqrt{m(m - 1)(m + \delta)(m + \delta + 1)}}{\delta} > 1
\end{equation}
is the exponential growth rate of the local weak limit of $\PA_n^{(m,\delta)}$, as identified in~\cite{HHR23}. Precisely,~\cite{VZ25} established the following result:

\begin{proposition}[Theorem 1.1, \cite{VZ25}]{\label{prop-typical-distance}}
Fix any $m\ge 2$ and $\delta>0$.
For $G_n\sim \PA_n^{(m,\delta)}$ and $u_n,v_n$ sampled from $V(G_n)$ uniformly and independently at random, it holds that
    \[
    \frac{\operatorname{dist}_{G_n}(u_n,v_n)}{\log_\nu n}\stackrel{\text{in probability}}{\rightarrow}1\,,\quad\text{as }n\to\infty\,.
    \]
\end{proposition}
It is further conjectured in \cite{VZ25} that typically, the diameter of $G_n\sim \PA_n^{(m,\delta)}$ is also $(1+o(1))\log_\nu n$ (note that Proposition~\ref{prop-typical-distance} already provides the lower-bound). We provide a proof of the conjectural upper-bound building on Proposition~\ref{prop-typical-distance}.

\subsection{Proof strategy}\label{subsec-intuition}
Our proof builds on a framework that converts a probabilistic upper bound on the typical distance in a random graph into a probabilistic upper bound on its diameter, at the cost of an additional additive term that is typically small.

To be more precise, let $M_n$ be an a.a.s.\ (asymptotically almost surely) upper bound on the median distance of $G_n$, meaning that with high probability over $G_n\sim\pa$,
\[
\mathbb{P}\left[\operatorname{dist}_{G_n}(u_n,v_n)\le M_n\right]\ge \frac{1}{2}\,,
\]
where $\mathbb{P}$ is taken over uniformly chosen vertices $u_n, v_n \in V_n$. The key observation is that if the $r$-neighborhoods of vertices in $G_n$ grow sufficiently rapidly in $r$, then for any pair of vertices $u,v$, with overwhelming probability there exist vertices in their respective small neighborhoods whose distance is at most $M_n$. This in turn implies that the diameter is at most $M_n$ plus a small additive error, provided that we have uniform growth estimates for the neighborhoods of all vertices.

Specializing to the setting of the preferential attachment model, in order to show that the gap between the median distance and the diameter of $G_n \sim \PA_n^{(m,\delta)}$ is $o(\log n)$, it suffices to show that with high probability, there exists some $R_n = o(\log n)$ such that the $R_n$-neighborhood of every vertex in $G_n$ has size $\omega(\log n)$. Assuming this holds, the above heuristics can be made rigorous via a sprinkling argument, yielding an a.a.s.\ upper bound of $M_n + O(R_n)$ on the diameter of $G_n$.

We provide several remarks on this approach. First, it relies only on soft arguments about random graphs and is fairly general. Moreover, it appears to be tight in several senses. On the one hand, for the preferential attachment model, while in Section~\ref{sec-neighborhood-size} we prove that one can take $R_n = O((\log n)^{2/3})$, we in fact expect that $R_n = O(\log\log n)$ suffices. This suggests an $O(\log\log n)$ gap between the typical distance and the diameter of $G_n \sim \PA_n^{(m,\delta)}$, which we believe to be tight by comparison with the behavior of random $d$-regular graphs. On the other hand, we note that for a sparse \ER graph, the gap between the diameter and the typical distance is indeed $\Theta(\log n)$,\footnote{More precisely, for a supercritical sparse \ER graph $\mathcal{G}(n,\lambda/n)$ with $\lambda>1$, a.a.s.\ the typical distance and the diameter of its giant component are $(c_1+o(1))\log n$ and $(c_2+o(1))\log n$, respectively, for two different constants $c_1 < c_2$ depending on $\lambda$.} partly because, with high probability, there exist vertices whose $\Theta(\log n)$-neighborhoods have size only $O(\log n)$ (e.g., leaves of large trees dangling from the giant $2$-core). In view of these observations, we believe that this approach may have broader and further applicability.

\subsection{Precise model definition and notations}\label{subsec-notations}

We now present the precise definition of the preferential attachment model we address. 
%\zh{I have slightly modified the definition here, mainly by adding labels to the edges.}\SG{sorry, where is the modification? It seems that this is the the same definition as before?}
\begin{definition}[Preferential attachment model]\label{def-preferential-attachment-model}
Given $m,n\in \mathbb{N}$, $\delta>-m$, and a set $V_n$ with cardinality $n$, an undirected graph $G_n$ with vertex set $V_n=\{v_1,\dots,v_n\}$ is defined as follows:
\begin{itemize}
% \item We generate the sequence of vertex arrival times, denoted by the bijection $\sigma: V_n \to [n]$, uniformly at random;
\item The initial graph $G_2=(V_2, E_2)$ consists of two vertices $v_1$ and $v_2$ and $m$ multiple edges labeled with $1,2,\ldots,m$ connecting them; 
% \HD{I changed the definition here because we need $\mathsf{deg}(v)\ge m$ to define the probability in (1).}
\item For $3 \le t \le n$, the graph $G_t=(V_t, E_t)$ is obtained by adding to $G_{t-1}$ a new vertex $v_t$ and connecting $m$ edges labeled with $1,2,\ldots,m$ from $v_t$ to vertices in $V_{t-1}$. 
Specifically, we construct a graph sequence $G_{t,0},G_{t,1},\ldots,G_{t,m}$ starting from $G_{t,0}=G_{t-1}$ and ending at $G_{t,m}=G_t$. For $1\leq i \leq m$, the graph $G_{t,i}$ is obtained by adding an edge labeled $i$ from $v_t$ to a vertex $v_{t,i} \in V_{t-1}$ with probability 
\begin{align}\label{eq-prob-def}
    \pa \left[ v_{t,i}=v_k \mid G_{t,i-1} \right] = \frac{D_{v_k}(t,i-1)+\delta}{ \sum_{\ell\leq t-1 } \left(  D_{v_\ell}(t,i-1) +\delta \right) }, \quad \forall 1\le k \leq t-1\,,
\end{align}
where $D_v(t,i)$ is the degree of $v$ in $G_{t,i}$. For simplicity, we write $D_v(t,m)$ as $D_v(t)$.
\end{itemize}
The graph $G_n$ is called a preferential attachment graph and we denote its distribution by $\PA_n^{(m,\delta)}$. It aligns with $\pa(d)$ in \cite{VZ25}.
\end{definition}
We remark that several slightly different definitions of the preferential attachment model exist in the literature; however, the specific version we choose does not substantially affect our results. 
%Indeed, as pointed out in~\cite{VZ25}, Proposition~\ref{prop-typical-distance} holds for all versions of the preferential attachment model, and our argument for Theorem~\ref{thm-main} is also applicable to any of these variants. 
Indeed, the proof of Lemma~\ref{lem-edgeset-cond-prob} is the only place where a specific version of $\pa$ is required, and this can be generalized to other variations (for an intuition behind the proof, see, e.g., \cite[Appendix C]{VZ25}).

%\SG{Haodong, please feel free to introduce any notation you need for the proof here.}
Throughout the remainder of the paper, we call $v_i$ the vertex with label $i$, $1 \le i \le n$, and we use $\llbracket a,b\rrbracket$ to denote the set of vertices with labels in the interval $[a,b]$. We write $\operatorname{N}_r(v)$ for the $r$-neighborhood of a vertex $v$ in $G_n$. Let $\operatorname{N}_r^{\downarrow}(v)$ denote the set of $\le r$-generation ascendants of $v$ in $G_n$ (i.e., $\operatorname{N}_r^{\downarrow}(v)$ includes all the vertices $u$ with labels less than $v$ for which there exists a path $P$ from $u$ to $v$ of length at most $r$ and with increasing labels). For any connected graph $H$ and any vertices $u,v \in V(H)$, we denote $\operatorname{dist}_{H}(u,v)$ to be the graph distance of $u,v$ in $H$. In addition, for any $u \in V(H)$ and $A \subset V(H)$ we denote $\operatorname{dist}_H(u,A)=\min\{ \operatorname{dist}_H(u,v): v \in A \}$.

%\HD{Add the notation about the edges}
%\SG{the definitions here are not useful anymore? possibly we only need to use the $N^{\downarrow}$ type neighborhood? Also, we could delete $\widehat{G}_n$} Moreover, as hinted earlier, we define $\widehat{G}_n$ as the subgraph of $G_n$ obtained by deleting the edges $e_{t,j}$ with $n/2+1\le t\le n$ and $2\le j\le m$, and we let $\widehat{\operatorname{N}}_r(v)$ be the $r$-neighborhood of $v$ in $\widehat{G}_n$. Finally, for any graph $H$, we denote by $\Delta(H)$ its max-degree. 

%\SG{Better to write a math definition here}.

%$\operatorname{N}_r^{\downarrow}(v)$ denote the set of $\le r$-generation ascendants of $v$ in $G_n$ (i.e., $\operatorname{N}_r^{\downarrow}(v)$ includes all the vertices $u$ with labels less than $v$ for which there exists a path $P$ from $u$ to $v$ of length at most $r$ and with increasing labels). 

\section{Uniform growth of the neighborhood size}\label{sec-neighborhood-size}

A key challenge in analyzing the preferential attachment model is the absence of independence across edges and the complexity of handling the resulting correlations. However, throughout our proof, we require only basic conditional probability estimates, as incorporated in the next lemma.

\begin{lemma}\label{lem-edgeset-cond-prob}
    Let $E$ and $E'$ be two sets of potential edges in $G_n\sim \pa$ such that $E\cap E'=\emptyset$. Assume that $V(E')\subset \llb s,n\rrb$, then 
    %\HD{Do we need to assume $V(E)\subset \llb s,n\rrb$ as well?} \zh{No. By the way, I added something in \eqref{eq-conditional-prob-edge-set} and \eqref{eq-conditional-prob-edge}. I think it should be fine?}
    \begin{equation}\label{eq-conditional-prob-edge-set}
    \PA_n^{(m,\delta)}[E'\cap E(G_n)\neq \emptyset\mid E\subset E(G_n)]\le \frac{|E'|(m+\delta+1)+|E|}{(2s-2)m+s\delta}\,.
    \end{equation}
    Specifically, taking $E'=\{e'\}$, we have
    \begin{equation}\label{eq-conditional-prob-edge}
    \PA_n^{(m,\delta)}[e'\in E(G_n)\neq \emptyset\mid E\subset E(G_n)]\le \frac{m+\delta+1+|E|}{(2s-2)m+s\delta}\,.
    \end{equation}
\end{lemma}

\begin{proof}
We write $E=\left\{(\ell_h,i_h,j_h):h\in [H]\right\}$ and $E'=\left\{(\ell_h',i_h',j_h'):h\in [H']\right\}$, where a triple $(\ell,i,j)$ means that there is an edge labeled $i$ between vertices $\ell$ and $j$ with $\ell>j$, that is, $v_{\ell,i}=v_j$ in \eqref{eq-prob-def}. Under this notation, define
    \begin{align*}
        p_s^E=\sum_{r=1}^H \mathds{1}\left\{s=j_r\right\}\quad\mbox{and}\quad
        q_s^E=\sum_{r=1}^H \mathds{1}\left\{\ell_r<s<j_r\right\} \,.
    \end{align*}
    Then, it follows directly from the combination of \cite[Equations (2.6), (2.7) and (3.3)]{VZ25} that
    \begin{align}\label{eq-edgeset-prob}
        \pa\left[E\subset E(G_n)\right]=\prod_{s=2}^n \frac{(m+\delta+p_s^E-1)_{p_s^E}((2s-3)m+(s-1)\delta+q_s^E-1)_{q_s^E}}{((2s-2)m+s\delta+p_s^E+q_s^E-1)_{p_s^E+q_s^E}} \,,
    \end{align}
    where $(x)_r=x(x-1)\ldots(x-r+1)$. Let $E_h=E\cup\{(\ell_h',i_h',j_h')\}$ for $h\in [H']$. Analogous to \eqref{eq-edgeset-prob}, if there does not exits $r\in[H]$ such that $(\ell_r,i_r)=(\ell_h',i_h')$ (note that $(\ell_r,i_r,j_r)\neq (\ell_h',i_h',j_h')$ since $E\cap E'=\emptyset$), then
    \begin{align}\label{eq-edgeset-prob-2}
        &\pa\left[E_h\subset E(G_n)\right]\nonumber\\
        =&\prod_{s=2}^n \frac{(m+\delta+p_s^{E_h}-1)_{p_s^{E_h}}((2s-3)m+(s-1)\delta+q_s^{E_h}-1)_{q_s^{E_h}}}{((2s-2)m+s\delta+p_s^{E_h}+q_s^{E_h}-1)_{p_s^{E_h}+q_s^{E_h}}} \,,
    \end{align}
    where
    \begin{align*}
        p_s^{E_h}=\sum_{r=1}^H \mathds{1}\left\{s=j_r\right\}+\mathds{1}\left\{s=j_h'\right\}\quad\mbox{and}\quad
        q_s^E=\sum_{r=1}^H \mathds{1}\left\{\ell_r<s<j_r\right\}+\mathds{1}\left\{\ell_h'<s<j_h'\right\} \,,
    \end{align*}
    otherwise $\pa\left[E_h\subset E(G_n)\right]=0$.
    %\SG{Please double check if this is correct. I can understand the meaning here, but there are two circumstances that may happen: (i) there exists $r$ such that $(\ell_r,i_r)=(\ell'_h,i'_h)$ but $j_r\neq j'_h$, this is of probability zero, as you said; (ii) there exists $r$ such that $(\ell_r,i_r)=(\ell'_h,i'_h)$ and $j_r = j'_h$, this implies the probability of $E_h$ and $E$ are equal (but not necessarily zero)? Anyway, this does not affect the proof of this lemma, but better rephrase, because actually we don't need to discuss the case ``there exists $r$ such that $(\ell_r,i_r)=(\ell'_h,i'_h)$".}\zh{The second case is excluded because $E\cap E'=\emptyset$.}
    Note that $j_h'\geq s$. Combining \eqref{eq-edgeset-prob} and \eqref{eq-edgeset-prob-2}, we conclude that
    \begin{align}
        &\pa\left[(\ell_h',i_h',j_h')\in E(G_n)\mid E\subset E(G_n)\right]\nonumber\\
        \leq\ &\frac{m+\delta+p_{j_h'}^{E_h}}{(2j_h'-2)m+j_h'\delta+p_{j_h'}^{E_h}+q_{j_h'}^{E_h}}\prod_{s=j_h'+1}^{\ell_h'-1} \frac{(2s-3)m+(s-1)\delta+q_s^{E_h}}{(2s-2)m+s\delta+p_s^{E_h}+q_s^{E_h}}\nonumber\\
        \leq\ &\frac{m+\delta+p_{j_h'}^{E_h}}{(2s-2)m+s\delta} \,.  \label{eq-one-edge-cond-prob}
    \end{align}
    Since 
    $$\sum_{h\in [H']}p_{j_h'}^{E_h}= \sum_{h\in [H']}\sum_{r\in[H]} \mathds{1}\left\{j_h'=j_r\right\}+\sum_{h\in [H']}\mathds{1}\left\{j_h'=j_h'\right\}\leq |E|+|E'|,$$ 
    the desired result follows directly from applying a union bound to \eqref{eq-one-edge-cond-prob}.
\end{proof}

% \begin{lemma}
%     Let $E=\left\{(\ell_h,i_h,j_h):h\in [H]\right\}$ be a set of edges in $\pa$ of size $H$, where $(\ell_h,i_h,j_h)$ means that the $i_h$ out-edge of vertex $\ell_h$ is connected to vertex $j_h$. Define
%     \begin{align*}
%         p_s^E=\sum_{r=1}^H \mathds{1}\left\{s=j_r\right\}\qquad\mbox{and}\qquad
%         q_s^E=\sum_{r=1}^H \mathds{1}\left\{\ell_r<s<j_r\right\}.
%     \end{align*}
%     Then,
%     \begin{align*}
%         \pa\left[E\right]=\prod_{s=2}^n \frac{(m+\delta+p_s^E-1)_{p_s^E}((2s-3)m+(s-1)\delta+q_s^E-1)_{q_s^E}}{((2s-2)m+s\delta+p_s^E+q_s^E-1)_{p_s^E+q_s^E}},
%     \end{align*}
%     where $(x)_r=x(x-1)\ldots(x-r+1)$.
% \end{lemma}
% \begin{proof}
%     It follows directly from the combination of \cite[equations (2.6), (2.7) and (3.3)]{VZ25}.
% \end{proof}

% \begin{corollary}
%     For $E=\left\{(\ell_h,i_h,j_h):h\in [H]\right\}$ and $E_0=\left\{(\ell_0,i_0,j_0)\right\}\not\subseteq E$,
%     \begin{align*}
%         \pa\left[E\cup E_0\mid E\right]=&\frac{m+\delta+p_{j_0}^E}{(2j_0-2)m+j_0\delta+p_{j_0}^E+q_{j_0}^E}\prod_{s=j_0+1}^{\ell_0-1} \frac{(2s-3)m+(s-1)\delta+q_s^E}{(2s-2)m+s\delta+p_s^E+q_s^E}\\
%         \leq &\frac{m+\delta+p_{j_0}^E}{(2j_0-2)m+j_0\delta+p_{j_0}^E+q_{j_0}^E}.
%     \end{align*}  

%     For $E_0\cap E= \emptyset$ such that any $(\ell,i,j)\in E_0$ satisfies $j\ge s$, it holds
%     \[
%     \PA_n^{(m,\delta)}[E_0\cap E(G_n)\neq \emptyset\mid E\subset E(G_n)]\le \frac{|E_0|(m+\delta)+|E|}{(2s-2)m+s\delta}\,.
%     \]
% \end{corollary}

Throughout the rest of the paper, we denote 
\begin{equation}{\label{eq-def-mathtt-L}}
    \mathtt{L}_n=(\log n)^{{2}/{3}} \,.
\end{equation}
The main goal of this section is to prove the following lemma, which provides a lower bound on the size of $O(\mathtt{L}_n)$-neighborhoods in ${G}_n$. The bound is far from tight but is sufficient for our use.

\begin{lemma}\label{lem-size-of-neighborhood}
    There exists $r=r(m,\delta)$ such that as $n\to\infty$, 
    \begin{equation}\label{eq-size-neighborhood-lower-bound-G}
    \PA_{n}^{(m,\delta)}\Big[ \big| \operatorname{N}_{r\mathtt{L}_n}(v) \big|\ge (\log n)^4\,,\forall v\in \llb 1,n\rrb\Big]=1-o(1)\,.
    \end{equation}
\end{lemma}

\begin{proof}
    First, using \cite[Theorem~8.33]{vdHof24} we have $\PA[\mathcal E_0]=1-o(1)$ that for some constant $C=C(m,\delta)>0$, where  
    \begin{equation}{\label{eq-def-mathcal-E-0}}
        \mathcal E_0:=\Big\{ \operatorname{diam}(G_{\lfloor e^{10\sqrt{\log n}}\rfloor}) \leq C\sqrt{\log n} \Big\} \,.
    \end{equation}
    Assume $\mathcal E_0$ holds, then it follows that for any $v\in \llb 1,e^{10\sqrt{\log n}}\rrb$, $|\operatorname{N}_{C\mathtt{L}_n}(v)|\ge e^{10\sqrt{\log n}}-1$ as the neighborhood contains all vertices in $\llb 1,e^{10\sqrt{\log n}}\rrb$. 
    %\zh{Is that because all vertices in $G_{e^{10\sqrt{\log n}}}$ are in $\operatorname{N}_{C\mathtt{L}_n}(v)$?} \HD{Yes.}
    Thus, under the event $\mathcal E_0$, if a vertex has distance no more than $R=2\mathtt{L}_n$ to the set $\llb 1,e^{10\sqrt{\log n}}\rrb$, then its $(C+2)\mathtt{L}_n$-neighborhood has size at least $e^{10\sqrt{\log n}}-1\gg (\log n)^4$. In what follows we prove that for any $u \in \llb 1,n\rrb$, we have %\SG{It should be a comma rather than ; ? ; usually means conditional probability? }
    \begin{equation}{\label{eq-goal-neighborhood-size}}
        \PA\Big[ \operatorname{N}_r^{\downarrow}(v) \cap \llb 1,e^{10\sqrt{\log n}}\rrb =\emptyset; |\operatorname{N}_r^{\downarrow}(v)| \leq (\log n)^4 \Big] \leq\tfrac{1}{n^3} \,.
    \end{equation}
    Provided that \eqref{eq-goal-neighborhood-size} is correct, we may consider the event
    \begin{equation}{\label{eq-def-mathcal-G-0}}
        \mathcal G_0 = \mathcal E_0 \bigcap \Big( \cap_{1 \leq u \leq n} \big\{ \operatorname{N}_{2\mathtt{L}_n}^{\downarrow}(u) \cap \llb 1,e^{10\sqrt{\log n}}\rrb =\emptyset; |\operatorname{N}_{2\mathtt{L}_n}^{\downarrow}(u)| \leq (\log n)^4 \big\}^{c} \Big) \,.
    \end{equation}
    We have that $\mathcal G_0$ implies that $|\operatorname{N}_{(C+2)\mathtt{L}_n}(u)|\geq (\log n)^4$ for all $u \in \llb 1,n \rrb$ and $\PA[\mathcal G_0]=1-o(1)$ from a union bound.
    
    To prove \eqref{eq-goal-neighborhood-size}, we consider the breath-first-search (BFS) process starting at $v$, and let $\operatorname{S}_r= \operatorname{N}_{r}^\downarrow(v)\setminus N_{r-1}^{\downarrow}(v)$ for $1\le r\le 2\mathtt{L}_n$ (we use the convention that $\operatorname{N}_0^{\downarrow}(v)=\{v\}$). Under our assumption, we have 
    \begin{equation}{\label{eq-assumption-operator-S}}
        \operatorname{S}_1\cup\cdots\cup\operatorname{S}_{2\mathtt{L}_n}\subset \llb e^{10\sqrt{\log n}},n\rrb \mbox{ and } \sum_{r\le 2\mathtt{L}_n}|\operatorname{S}_r|\le (\log n)^2\,.
    \end{equation} 
    We claim that for any $r$, given any choices of $k_r=|\operatorname{S}_r|,0\le r\le 2\mathtt{L}_n$ (with $k_0=1$), it holds that
    \begin{align}
        &\pa\Big[|\operatorname{S}_r|=k_r,\operatorname{S}_r \subset \llb e^{10\sqrt{\log n}},n\rrb \mbox{ for all } 1\le r\le 2\mathtt{L}_n \Big] \nonumber \\
        \le\ & \prod_{r=1}^{2\mathtt{L}_n} (mk_{r-1})^{mk_{r-1}-k_{r}}e^{-3\sqrt{\log n}(mk_{r-1}-k_{r})}\,. \label{eq-prob-upper-bound}
    \end{align}
    The claim follows upon showing the following conditional probability estimate: for any $1\le r\le 2\mathtt{L}_n$ and any legitimate realizations of $\operatorname{S}_0,\dots,\operatorname{S}_{r-1}$ such that $|\operatorname{S}_i|=k_i$ and $\operatorname{S}_i \subset \llb e^{10\sqrt{\log n}},n\rrb$ we have 
    \begin{equation}\label{eq-conditional-probability-upper-bound}
        \pa\Big[ |S_r|=k_r\mid \operatorname{S}_0,\dots,\operatorname{S}_{r-1} \Big]\le (mk_{r-1})^{mk_{r-1}-k_{r}}e^{-3\sqrt{\log n}(mk_{r-1}-k_{r})}\,.
    \end{equation}
    Clearly $k_{r}\le mk_{r-1}$ and if equality holds there is nothing to prove. Otherwise, we have $mk_{r-1}-k_{r}$ edges among the $mk_{r-1}$ edges attached from vertices in $\operatorname{S}_{r-1}$ attaches to some other vertices in $\operatorname{S}_1 \cup \ldots \cup \operatorname{S}_{r-1}$. The choices of such $mk_{r-1}-k_{r}$ edges are at most $(mk_{r-1})^{mk_{r-1}-k_{r}}$. For each fixed choice of these edges, we label them as $e_1,\dots,e_{mk_{r}-k_{r-1}}$ according to their labels with increasing order. In addition, we denote $E_i$ to be the set of edges attached from vertices in $\operatorname{S}_0 \cup \ldots \cup \operatorname{S}_{r-1}$ that occurs prior to $e_i$ in the BFS process. Then we have $|E_i| \leq m|\operatorname{S}_0 \cup \ldots \cup \operatorname{S}_{r-1}|\leq m(\log n)^4$.
    Using Lemma~\ref{lem-edgeset-cond-prob}, conditioned on $E_{i}$ we have that the probability that $e_k$ attaches to an existing vertex in $\operatorname{S}_0 \cup \ldots \cup \operatorname{S}_{r-1}$ is at most (recall that we choose $\mathtt{L}_n=(\log n)^{2/3}$)
    \[
        \frac{|E_i|+(m+\delta)|\operatorname{S}_0 \cup \ldots \cup \operatorname{S}_{r-1}|}{ 2(e^{10\sqrt{\log n}}-2)m + e^{10\sqrt{\log n}}\delta } \leq e^{-3\sqrt{\log n}} \,.
    \]
    Therefore, we obtain the probability upper bound $e^{-3\sqrt{\log n}(mk_{r-1}-k_{r})}$ and thus \eqref{eq-conditional-probability-upper-bound} follows by taking the union bound. This proves \eqref{eq-prob-upper-bound}. 

    Note that the right hand side of \eqref{eq-prob-upper-bound} can be further relaxed to
    \begin{align*}
        \prod_{r=1}^{2\mathtt{L}_n}\big(m(\log n)^2e^{-3\sqrt{\log n}}\big)^{mk_{r-1}-k_r}\le&\ \exp\Big(-2\sqrt{\log n}\sum_{r=1}^{2\mathtt{L}_n}(mk_{r-1}-k_r)\Big)\\
        \le&\ \exp(-2\mathtt{L}_n\sqrt{\log n}/2\log\log n)\,,
    \end{align*}
    where the last inequality is due to the fact that $k_r\ge 1$ (as the vertex in $\operatorname{S}_{r-1}$ with the minimal label must attach to a vertex not in $\operatorname{S}_0\cup\cdots\operatorname{S}_{r-1}$) and $k_r \leq (\log n)^2$ for all $r \leq R$ implies that there does not exist any $1 \leq r \leq R$ such that $k_{r+i}=mk_{r+i-1}$ for all $1 \leq i \leq 2\log\log n$ (also recall that $mk_{r-1}-k_r \geq 0$). Using this bound and by further taking the union bound over the choices of $k_r=|\operatorname{S}_r|\le (\log n)^2$, we see the probability we concern is upper bounded by
    \[
        (\log n)^{2R}\exp(-\mathtt{L}_n\sqrt{\log n}/\log\log n)\le \exp(-\mathtt{L}_n\sqrt{\log n}/2\log\log n)\le n^{-3}\,,
    \]
    %\zh{Should it be $n^{-3}$ because of \eqref{eq-goal-neighborhood-size}?}\HD{Yes}
    as desired. This concludes the proof.
\end{proof}

\section{Proof of Theorem~\ref{thm-main}}

Let $\{M_n\}$ be an increasing sequence that a.a.s. upper bounds the medium distance of $G_n$. By Proposition~\ref{prop-typical-distance} we can pick $\{M_n\}$ such that $M_n=(1+o(1))\log_\nu n$.   
\begin{comment}
Besides Lemma~\ref{lem-size-of-neighborhood}, our main technical input is the following result on the distance between typical vertices in \cite{VZ25}.
\begin{proposition}[Theorem 1.1, \cite{VZ25}]
    There exists a sequence $\varepsilon_n\downarrow 0$ such that as $n\to\infty$,
    \[
    \PA_{n}^{(m,\delta)}\Big[\mathbb{P}_{(u,v)\sim \operatorname{Uni}(\llb 1,n\rrb)^{\otimes 2}}[\operatorname{dist}_{G_n}(u,v)\le M_n]\Big]\ge 1-\varepsilon_n\,.
    \]
\end{proposition}
\end{comment}
Our goal is to show that the diameter of $G_n$ has an a.a.s. upper bound $M_n+O(\mathtt{L}_n)$ (recall \eqref{eq-def-mathtt-L}). We further denote $\mathtt{K}_n =\tfrac{n}{\log n}$.

\begin{definition}\label{def-typical-vertex}
    For any vertex $u$ in $\llb 1,n-2\mathtt{K}_n\rrb$, denote 
    \begin{equation}{\label{eq-def-mathcal-A(u)}}
        \mathcal A(u):= \big\{ v \in \llb 1,n-2\mathtt{K}_n\rrb: \operatorname{dist}_{G_{n-2\mathtt{K}_n}}(u,v) \leq M_n\big\} \,.
    \end{equation}
    In addition, a vertex $u$ in $\llb 1,n-2\mathtt{K}_n\rrb$ is called typical, if $|\mathcal A(u)|\geq \lfloor n/10\rfloor$. 
\end{definition}

\begin{lemma}{\label{lem-mathcal-G-is-typical}}
    Define $\mathcal G_1$ as the event that there are at least $\lfloor n/10\rfloor$ typical vertices. We have $\PA_{n}^{(m,\delta)}[\mathcal G_1]=1-o(1)$ as $n\to\infty$.
\end{lemma}
\begin{proof}
    %\zh{I made some changes here (turn $\llb 1,n\rrb$ into $\llb1,n-2\mathtt{K}_n\rrb$).}
    Consider the event
    \begin{align*}
        \widetilde{\mathcal G}_1:= \Big\{ \mathbb{P}_{(u,v)\sim \operatorname{Uni}(\llb 1,n-2\mathtt{K}_n\rrb)^{\otimes 2}}\big[ \operatorname{dist}_{G_{n-2\mathtt{K}_n}}(u,v)\le M_n \mid G_n \big] \geq 1/2\Big\}\,,
    \end{align*}
    where $\mathbb{P}_{(u,v)\sim \operatorname{Uni}(\llb 1,n-2\mathtt{K}_n\rrb)^{\otimes 2}}$ is taken over $u,v$ chosen from $\llb 1,n-2\mathtt K_n\rrb$ uniformly and independently at random. 
    Then, by our choice of $M_n$ (which is increasing in $n$), $\PA_n^{(m,\delta)}[\widetilde{\mathcal G}_1]=1-o(1)$. 
    
    On the other hand, we claim that $\mathcal G_1^c\subset \widetilde{\mathcal G}_1^c$. This is because assuming $G_n\in \mathcal G_1^c$, by the union bound we have
    \begin{align*}
    &\ \mathbb{P}_{(u,v)\sim \operatorname{Uni}(\llb 1,n-2\mathtt{K}_n\rrb)^{\otimes 2}}\big[ \operatorname{dist}_{G_n}(u,v)\le M_n \mid G_n \big]\\
    \le\ & \mathbb{P}_{u\sim \operatorname{Uni}(\llb 1,n-2\mathtt{K}_n\rrb)}\big[u\text{ is typical} \mid G_n \big]+ \\
    &\mathbb{P}_{(u,v)\sim \operatorname{Uni}(\llb 1,n-2\mathtt{K}_n\rrb)^{\otimes 2}}\big[ u\text{ is not typical},\operatorname{dist}_{G_{n-2\mathtt K_n}}(u,v)\le M_n \mid G_n \big] \\
    \le&\ 0.1+0.1<1/2\,,
    \end{align*}
    %\zh{Maybe we should mention that we condition on the graph $G_n$ (or something else) here?}\HD{Does it look good now?}\zh{Shall we write $\mathbb{P}[u\text{ is typical}\mid G_n]$ instead of $\mathbb{P}[u\text{ is typical}]$? Since if we just write $\mathbb P$, it seems that we did not condition on the graph. And there might be some ambiguity, as we defined at the beginning of the proof that $\mathbb{P}$ is taken over $u,v$ chosen from $\llb 1,n-2\mathtt K_n\rrb$ uniformly and independently at random.}
    and thus $G_n\in \widetilde{\mathcal G}_1^c$. Therefore, we have $\PA_{n}^{(m,\delta)}[\mathcal G_1]\ge \PA_{n}^{(m,\delta)}[\widetilde{\mathcal G}_1]=1-o(1)$, as desired.
\end{proof}

\begin{lemma}{\label{lem-dist-small-vert}}
    Recall the definition of $\mathcal G_0$ in \eqref{eq-def-mathcal-G-0}. Also define 
    \begin{equation}{\label{eq-def-mathcal-G-2}}
        \mathcal G_2:=\cap_{1 \leq u,v \leq n-2\mathtt{K}_n}\Big\{ \operatorname{dist}_{G_n}(u,v) \leq M_n+2\mathtt{L}_n+4 \Big\} \,.
    \end{equation}
    We have $\PA_n^{(m,\delta)}[\mathcal G_2^{c};\mathcal G_1;\mathcal G_0]=o(1)$.
\end{lemma} 
\begin{proof}
    Our proof will follow a two-step argument. Denote $\mathcal T\subset \llb 1,n-2\mathtt{K}_n\rrb$ to be the set of typical vertices. Also fix $u,v \in \llb 1,n-2\mathtt{K}_n\rrb$. We first show that for all realizations $G_{n-2\mathtt{K}_n}$ that is compatible with $\mathcal G_0 \cap \mathcal G_1$ we have
    \begin{equation}{\label{eq-connect-to-typical-vert}}
        \PA_{n}^{(m,\delta)}\big[ \operatorname{dist}_{G_{n-\mathtt{K}_n}}(u,\mathcal T) \leq \mathtt{L}_n+2 \mid G_{n-2\mathtt{K}_n} \big] \geq 1-\tfrac{1}{n^3} \,.
    \end{equation}
    To this end, denote $\widehat{\operatorname{N}}_{\mathtt{L}_n}(u)$ to be the $\mathtt{L}_n$-neighborhood of $u$ in $G_{n-2\mathtt{K}_n}$ (note that $\widehat{\operatorname{N}}_{\mathtt{L}_n}(u)$ and $\mathcal T$ are measurable with $G_{n-2\mathtt{K}_n}$). Under $\mathcal G_0$ we get that $|\widehat{\operatorname{N}}_{\mathtt{L}_n}(u)| \geq (\log n)^3$. If $\widehat{\operatorname{N}}_{\mathtt{L}_n}(u) \cap \mathcal T\neq \emptyset$ then we have $\operatorname{dist}_{G_{n-2\mathtt{K}_n}}(u,\mathcal T)\leq \mathtt{L}_n$. Otherwise, we have
    \begin{align*}
        &\PA_n^{(m,\delta)}\Big[ \operatorname{dist}_{G_{n-\mathtt{K}_n}}(u,\mathcal T) \geq \mathtt{L}_n+2 \mid G_{n-2\mathtt{K}_n} \Big] \\
        \leq\ &\PA_n^{(m,\delta)}\Big[ \cap_{n-2\mathtt{K}_n+1\leq w \leq n-\mathtt{K}_n} \big( \{ \operatorname{N}_1^\downarrow(w) \cap \mathcal T=\emptyset \} \cup \{ \operatorname{N}_1^\downarrow(w) \cap \widehat{\operatorname{N}}_{\mathtt{L}_n}(u) = \emptyset \} \big) \mid G_{n-2\mathtt{K}_n} \Big] \,.
    \end{align*}
    %\zh{Sorry, is $\operatorname{N}^\downarrow(w)=\operatorname{N}^\downarrow_1(w)$? I did not find its definition.}
    Note that for any $n-2\mathtt{K}_n+1\leq w \leq n-\mathtt{K}_n$, given any realizations $G_{w-1}$ that is compatible with $\mathcal G_0 \cap \mathcal G_1$ we have
    \begin{align*}
        &\PA_n^{(m,\delta)}\big[ \big( \{ \operatorname{N}_1^\downarrow(w) \cap \mathcal T=\emptyset \} \cup \{ \operatorname{N}_1^\downarrow(w) \cap \widehat{\operatorname{N}}_{\mathtt{L}_n}(u) = \emptyset \} \big) \mid G_{w-1} \big] \\
        \leq\ & 1- \Omega(1)\cdot \tfrac{|\widehat{\operatorname{N}}_{\mathtt{L}_n}(u)|}{n} \leq 1-\Omega(1) \cdot \tfrac{(\log n)^3}{n} \,.
    \end{align*}
    %\zh{Why we have a $(1-\Omega(1))$ factor in the first inequality?} \red{()}
    Thus, we have
    \begin{align*}
        &\PA_n^{(m,\delta)}\big[ \operatorname{dist}_{G_{n-\mathtt{K}_n}}(u,\mathcal T) \geq \mathtt{L}_n+2 \mid G_{n-2\mathtt{K}_n} \big] \leq \big(1-\Omega(1)\cdot\tfrac{(\log n)^3}{n}\big)^{\mathtt{K}_n} \leq \tfrac{1}{n^3} \,,
    \end{align*}
    which verifies \eqref{eq-connect-to-typical-vert}. Now for any realization $G_{n-\mathtt{K}_n}$ compatible with $\mathcal G_0,\mathcal G_1$ and that there exists $w \in \mathcal T$ with $\operatorname{dist}_{G_{n-\mathtt{K}_n}}(u,w)\leq \mathtt{L}_n+2$. Since $w \in \mathcal T$ we have $|\mathcal A(w)|\geq \lfloor n/10 \rfloor$. Similarly as \eqref{eq-connect-to-typical-vert}, we can show that 
    \begin{equation}{\label{eq-connect-to-typical-neighborhood}}
        \PA_n^{(m,\delta)}\big[ \operatorname{dist}_{G_n}(v,\mathcal A(w)) \leq \mathtt{L}_n+2 \mid G_{n-\mathtt{K}_n} \big] \geq 1-\tfrac{1}{n^3} \,. 
    \end{equation}
    Combined with \eqref{eq-connect-to-typical-vert}, it holds that $\PA_n^{(m,\delta)}[\operatorname{dist}_{G_n}(u,v) \leq M_n+2\mathtt{L}_n+4] \geq 1-\tfrac{2}{n^3}$. The desired result then follows from a simple union bound.
\end{proof}

\begin{lemma}\label{lem-many-large-label-in-neighborhood}
    Define  
    \begin{equation}{\label{eq-def-mathcal-G-3}}
        \mathcal G_3 := \cap_{n-2\mathtt{K}_n+1 \leq u \leq n} \Big\{ \operatorname{dist}_{G_n}(u,\llb 1,n-2\mathtt{K}_n \rrb) \leq \mathtt{L}_n+1 \Big\} \,,
    \end{equation}
    we have $\PA_n^{(m,\delta)}[\mathcal G_3^c;\mathcal G_0]=o(1)$.
\end{lemma}
\begin{proof}
    Fix any $n-2\mathtt{K}_n+1\leq u \leq n$ and perform BFS in $\operatorname{N}_{\mathtt{L}_n}^\downarrow(u)$. Suppose that the first $\mathtt M=(\log n)^3$ vertices are $u_1,\ldots,u_{\mathtt M}$ (we list them in BFS order). Under $\mathcal E_0$ we have $\operatorname{dist}_{G_n}(u_i,u)\leq \mathtt{L}_n$. Thus, we see that
    \begin{align*}
        &\PA_n^{(m,\delta)}\big[ \operatorname{dist}(u,\llb 1,n-2\mathtt{K}_n \rrb) \geq \mathtt{L}_n+1; \mathcal G_0 \big] \\
        \leq\ &\PA_n^{(m,\delta)}\big[ u_k>n-2\mathtt{K}_n, u_k \mbox{ does not attach to } \llb 1,n-2\mathtt{K}_n \rrb \mbox{ for all } 1 \leq k \leq \mathtt M \big] \,.
    \end{align*}
    In addition, denote $H_k$ to be the graph induced by all the attachment edges of $u,u_1,\ldots,u_k$. Then we have $H_k \subsetneq H_{k+1}, |E(H_k)|=m(k+1)$ and $u_{k+1} \in V(H_k)$ is determined by $H_k$. %\zh{Do you mean that given $H_k$, the probability that $u_{k+1}=v$ is independent of $H_0,\ldots,H_{k-1}$?} \ZL{I mean that once you know all $(\ell,i,j)\in H_k$ there is an unique choice of $u_{k+1}$ (which is the vertex that has the first BFS order in $V(H_k) \setminus \{ u_1,\ldots,u_k \}$)}. 
    Note that conditioned on any realization $H_{k-1}$ such that $u_i$ does not attach to $\llb 1,n-2\mathtt{K}_n \rrb$ for all $1 \leq i \leq k-1$, we have
    \begin{align*}
        &\PA_n^{(m,\delta)}\big[ u_k \mbox{ does not attach to } \llb 1,n-2\mathtt{K}_n \rrb \mid H_{k-1} \big] \\
        \leq\ &\PA_n^{(m,\delta)}\Big[ \big( \cup_{n-2\mathtt{K}_n+1 \leq w \leq u_k} \{ (u_k,1,w) \} \big) \cap E(G_n) \neq \emptyset \mid H_{k-1} \subset E(G_n) \Big] \\
        \leq\ &\frac{2\mathtt{K}_n(m+\delta+1)+|E(H_{k-1})|}{(2(n-2\mathtt{K}_n)-2)m+(n-2\mathtt{K}_n)\delta} \leq O(1) \cdot \tfrac{\log n}{n} \,,
    \end{align*}
    where the second inequality follows from Lemma~\ref{lem-edgeset-cond-prob}. Thus, we get that
    \begin{align*}
        \PA_n^{(m,\delta)}\big[ u_k>n-2\mathtt{K}_n, u_k \mbox{ does not attach to } \llb 1,n-2\mathtt{K}_n \rrb \mbox{ for all } 1 \leq k \leq \mathtt M \big] \leq \Big( \tfrac{O(\log n)}{n} \Big)^{\mathtt M} \leq \tfrac{1}{n^2} \,, 
    \end{align*}
    and the desired result follows from a simple union bound.
\end{proof}

We can now finish the proof of Theorem~\ref{thm-main}.

\begin{proof}[Proof of Theorem~\ref{thm-main}]
    Recall \eqref{eq-def-mathcal-G-2} and \eqref{eq-def-mathcal-G-3}. It is clear that $\mathcal G_2 \cap \mathcal G_3$ implies that $\operatorname{diam}(G_n) \leq M_n+4\mathtt{L}_n+6=(1+o(1))\log_\nu n$. Additionally, we have
    \begin{align*}
        &\PA_n^{(m,\delta)}\big[ \operatorname{diam}(G_n)>M_n+4\mathtt{L}_n+6 \big] \leq \PA_n^{(m,\delta)}[ \mathcal G_2^c \cup \mathcal G_3^c ] \\
        \leq\ &\PA_n^{(m,\delta)}[\mathcal G_0^c] + \PA_n^{(m,\delta)}[\mathcal G_1^c;\mathcal G_0] + \PA_n^{(m,\delta)}[\mathcal G_2^c;\mathcal G_0;\mathcal G_1] + \PA_n^{(m,\delta)}[\mathcal G_3^c;\mathcal G_0] =o(1) \,,
    \end{align*}
    where the second equality follows from a combination of Lemmas~\ref{lem-size-of-neighborhood}, \ref{lem-mathcal-G-is-typical}, \ref{lem-dist-small-vert} and \ref{lem-many-large-label-in-neighborhood}. This proves the desired upper-bound, and the lower bound follows from Proposition~\ref{prop-typical-distance}. 
\end{proof}

\section*{Acknowledgment}
We would like to thank Remco van der Hofstad for introducing this interesting problem.

H. Du is partially supported by an NSF-Simons research collaboration grant (award number 2031883).
S. Gong and Z. Li are partially supported by National Key R\&D program of China (No. 2023YFA1010103) and NSFC Key Program (Project No. 12231002). H. Zhu is supported by the EU Horizon 2020 programme (Marie Skłodowska-Curie grant No. 945045), the NWO Gravitation project NETWORKS (No. 024.002.003) and the National Science Foundation (No. DMS-1928930).

%\zh{There is a small issue in the references. We could possibly ignore it, but Remco prefers the sorting to be based on 'H' rather than 'r'. I tried my previous method to solve this, but it did not work.}
\small
\bibliographystyle{alpha}
\bibliography{ref}

\end{document}